%% file: main_3.tex
\documentclass[a4paper,11pt]{article}

\usepackage[utf8]{inputenc}
\usepackage[english]{babel}
\usepackage{amsmath,amssymb,amsthm,amsfonts,amsbsy,bbm}
\usepackage{hyperref,fullpage}
\usepackage[capitalise]{cleveref}

\usepackage{multicol}

\usepackage{tikz}
\usetikzlibrary{patterns,decorations}
\usetikzlibrary{calc}
\DeclareMathSymbol{\shortminus}{\mathbin}{AMSa}{"39}

\allowdisplaybreaks
\usepackage[format=plain,
            labelfont={bf,it},
            textfont=it]{caption}
\crefname{equation}{}{}

\newtheorem{theorem}{Theorem}[section]
\newtheorem{lemma}[theorem]{Lemma}
\newtheorem{definition}[theorem]{Definition}

\newtheorem{corollary}[theorem]{Corollary}

\newtheorem{proposition}[theorem]{Proposition}
\newtheorem{example}[theorem]{Example}

\newcommand{\mT}{\mathcal{T}}

\newcommand{\mG}{\mathcal{G}}

\newcommand{\mP}{\mathcal{P}}

\newcommand{\parr}[1]{\left({#1}\right)}
\newcommand{\abs}[1]{\left\lvert{#1}\right\rvert}
\newcommand{\floor}[1]{\left\lfloor{#1}\right\rfloor}
\newcommand{\ceil}[1]{\left\lceil{#1}\right\rceil}

\newcommand{\ao}{%
  \:\mathrel{%
    \vcenter{\offinterlineskip
      \ialign{##\cr$\alpha$\cr\noalign{\kern1.7pt}$\omega$\cr} %
    }%
  }%
  \!
}

\usepackage{todonotes} 

\title{Independent Chains in Acyclic Posets}
\author{
	Nika Salia\thanks{Alfr\'ed R\'enyi Institute of Mathematics, Budapest. Central European University, Budapest. E-Mail: {\tt salia.nika@renyi.hu}.} \and
	Christoph Spiegel\thanks{Universitat Polit\`ecnica de Catalunya, Department of Mathematics, Edificio Omega, 08034 Barcelona, Spain, and Barcelona Graduate School of Mathematics. E-mail: {\tt christoph.spiegel@upc.edu}.} \and
	Casey Tompkins\thanks{Karlsruhe Institute of Technology, Germany.  Discrete Mathematics Group, Institute for Basic Science (IBS), Daejeon, Republic of Korea. E-Mail: {\tt ctompkins496@gmail.com}.}  \and
	Oscar Zamora\thanks{Central European University, Budapest. Universidad de Costa Rica, San Jos\'e.E-Mail: {\tt oscar.zamoraluna@ucr.ac.cr.}}
}
\date{\today}

\begin{document}

\maketitle

\begin{abstract}
	We consider the problem of determining the maximum order of an induced vertex-disjoint union of cliques in a graph. More specifically, given some family of graphs $\mG$ of equal order, we are interested in the parameter
 	\begin{equation*}
 		\ao(\mG) = \min_{G \in \mG} \max \big\{ |U| : U \subseteq V, \:G[U] \text{ is a vertex-disjoint union of cliques} \big\}.
 	\end{equation*}
	We determine the value of this parameter precisely when $\mG$ is the family of comparability graphs of $n$-element posets with acyclic cover graph. In particular, we show that $\ao(\mG) = (n+o(n))/\log_2 (n)$ in this class.
\end{abstract}

\section{Introduction}

Given a finite graph $G = (V,E)$, we consider the parameter
\begin{equation*}
	\ao(G) = \max \big\{ |U| : U \subseteq V, \:G[U] \text{ is a vertex-disjoint union of cliques} \big\}.
\end{equation*}
Trivially the bound
\begin{equation} \label{eq:ao}
	\ao(G) \geq \max \big( \alpha(G), \omega(G) \big)
\end{equation}
holds for any graph $G$. However, the disjoint union of $k \geq 1$ cliques of size $k$ provides an example on $k^2$ vertices where $\ao(G) = k^2$ is much greater than both the independence number $\alpha(G) = k$ and the clique number $\omega(G) = k$. In the other direction, it is obvious that $\ao(G) \le \alpha(G) \, \omega(G)$.

This parameter was first introduced in this form by Ertem et al.~\cite{Ertem}, but it can be shown to be equivalent to the \textsc{Cluster Vertex Deletion} problem~\cite{Huffner-etal-2010}.
There is a large volume of literature on the computational aspects of related parameters, see for example~\cite{Shamir}. In particular, it follows from the results in~\cite{Lewis}, that the problem of determining $\ao(G)$ is NP-hard in general. We take a different approach by investigating the extremal properties of the parameter.  In particular, we are interested in minimising $\ao(G)$ across graphs $G$ belonging to a certain class of graphs with a given number of vertices. The particular class of graphs we want to study is that of comparability graphs of posets.

\begin{definition} 
	A \emph{poset} consists of a finite set $P$ with some partial ordering $<$. We write $p_1 \nsim p_2$ when two elements $p_1,p_2 \in P$ are \emph{incomparable}, that is neither $p_1 < p_2$ nor $p_2 < p_1$. We also say that $p_1$ \emph{covers} $p_2$, if $p_1 > p_2$ and there is no $p_3 \in P$ such that $p_1 > p_3 > p_2$. There are three different graphs or diagrams commonly associated with a poset $P$, all of them on the vertex set $P$:
	\begin{enumerate} \setlength\itemsep{0em}
		\item	The \emph{comparability graph} of $P$ is the graph obtained by connecting two elements by an edge if and only if they are comparable.
		\item	The \emph{cover graph} of $P$ is the graph obtained by connecting two elements by an edge if and only if one covers the other.
		\item	The \emph{Hasse diagram} of $P$ is an embedding of the cover graph in the plane where for any two points $x,y \in P$, $x$ is drawn higher than $y$ if $x$ covers $y$.
	\end{enumerate}
\end{definition}

Given a poset $P$, we will write $\ao(P)$ for $\ao(G)$, where $G$ is the comparability graph of the poset. Note that determining $\ao(P)$ is the same as determining the size of the largest independent collection of chains in the poset, meaning that any two elements from two different chains are incomparable. In fact, writing $h(P)$ for the \emph{height} of $P$, that is the size of its largest chain, and $w(P)$ for the \emph{width} of $P$, that is the size of its largest anti-chain, we have $\alpha(G) = w(P)$ and $\omega(G) = h(P)$. Thus,~\cref{eq:ao} becomes
\begin{equation} \label{eq:hw}
	\ao(P) \geq \max \big( w(P), h(P) \big).
\end{equation}
Note that the corresponding upper bound $\ao(P) \leq w(P) \, h(P)$ is trivial since $\abs{P} \leq w(P) \, h(P)$.

\begin{example}
	If $P_n$ is the Boolean lattice of subsets of a set of size $n$, then 
\begin{equation*}
	\ao(P_n) = 2\binom{n-1}{\lfloor (n-1)/2 \rfloor} = \Theta \big( 2^n/\sqrt{n} \big).
\end{equation*}
The lower bound is obtained through the sets $\{ A : A \subset [n], |A| = n/2 \}$ when $n$ is even and by $\{ A , A \cup \{n\} : A \subset [n-1], |A| = \floor{(n-1)/2} \}$ for arbitrary $n$. A matching upper bound was given by Ahlswede and Zhang~\cite{Ahlswede} and follows from the Bollob\'as inequality~\cite{Bollobas-1965}.
\end{example}

Given a family $\mathcal{F}$ of finite posets, we write 
\begin{equation*}
	\ao(\mathcal{F}) = \min_{P \in \mathcal{F}} \ao(P).
\end{equation*}
When $\mP_n$ is the family of all posets of size $n$, we have $\ao(\mP_n) = \ceil{\sqrt{n}}$. A lower bound follows from Dilworth's theorem:  the size of the maximal antichain equals to the size of a minimal chain cover of $P$. Hence by the pigeonhole principle $h(P) \geq n/w(P)$ and therefore $\ao(P) \geq \max \{n/w(P),w(P)\} \geq \sqrt{n}$ from \cref{eq:hw}. For a matching upper bound, consider the complete multipartite ordering with $\ceil{\sqrt{n}}$ parts, $\floor{\sqrt{n}}$ of which are of size $\floor{\sqrt{n}}$ and one which is of size $n - \floor{\sqrt{n}}^2$.

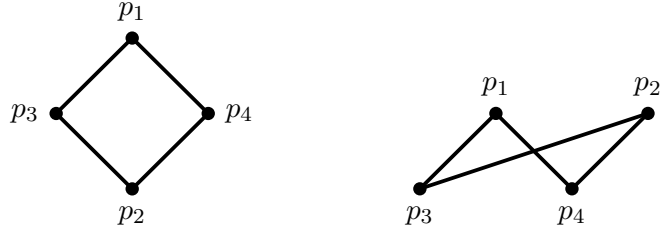
\begin{figure}[ht]
\begin{center}
%
	\begin{tikzpicture}
	\begin{scope}
		\draw[black, line width=0.5mm] (-1,0) -- (0,1);
		\draw[black, line width=0.5mm] (-1,0) -- (0,-1);
		\draw[black, line width=0.5mm] (0,1) -- (1,0);
		\draw[black, line width=0.5mm] (0,-1) -- (1,0);
		\node[circle,fill=black, inner sep=1pt, minimum size=5pt,label=above:{$p_1$}] (char) at (0,1) {};
		\node[circle,fill=black, inner sep=1pt, minimum size=5pt,label=below:{$p_2$}] (char) at (0,-1) {};
		\node[circle,fill=black, inner sep=1pt, minimum size=5pt,label=left:{$p_3$}] (char) at (-1,0) {};
		\node[circle,fill=black, inner sep=1pt, minimum size=5pt,label=right:{$p_4$}] (char) at (1,0) {};
	\end{scope}
	\end{tikzpicture}
%
\hspace{1.5cm}
%
	\begin{tikzpicture}
	\begin{scope}
		\draw[black, line width=0.5mm] (0,1) -- (1,0);
		\draw[black, line width=0.5mm] (1,0) -- (2,1);
		\draw[black, line width=0.5mm] (2,1) -- (-1,0);
		\draw[black, line width=0.5mm] (-1,0) -- (0,1);
		\node[circle,fill=black, inner sep=1pt, minimum size=5pt,label=above:{$p_1$}] (char) at (0,1) {};
		\node[circle,fill=black, inner sep=1pt, minimum size=5pt,label=above:{$p_2$}] (char) at (2,1) {};
		\node[circle,fill=black, inner sep=1pt, minimum size=5pt,label=below:{$p_3$}] (char) at (-1,0) {};
		\node[circle,fill=black, inner sep=1pt, minimum size=5pt,label=below:{$p_4$}] (char) at (1,0) {};
	\end{scope}
	\end{tikzpicture}
%
\end{center}
\caption{The two types of cycles that cannot occur in the Hasse diagram of an acyclic poset, with a cycle of type (1) to the left and a cycle of type (2) to the right.}  \label{fig:acyclic}
\end{figure}

There are other examples of posets where the parameter is of order $\sqrt{n}$. However, all of them rely on obtaining obstructions through cycles in the cover graph. This begs the question whether excluding such cycles allows one to obtain a significantly larger lower bound for the parameter. For this purpose, let us say that a poset $P$ is \emph{acyclic} if its cover graph is acyclic, that is there are no $p_1,p_2,p_3,p_4 \in P$ such that either (1) $p_1 > p_3,p_4 > p_2$ and $p_3 \nsim p_4$ or (2) $p_1,p_2 > p_3,p_4$, $p_3 \nsim p_4$ and $p_1 \nsim p_2$. See \cref{fig:acyclic} for an illustration of the two types of cycles. The following theorem gives a positive answer to the previous question.



\begin{theorem} \label{thm:acyclic}
	Let $\mathcal{T}_n$ denote the family of all acyclic posets of size $n$. We have
	\begin{equation*}
		\ao(\mathcal{T}_n) = \frac{n}{\log_2 n} \, \big( 1 + o(1) \big).
	\end{equation*}
\end{theorem}

We will in fact obtain a more precise formula for the inverse of the problem, that is for a given parameter $a$ that is a power of $2$ we will determine the exact largest cardinality of an acyclic poset $P$ with $\ao(P) = a$. In particular, when $a = 2^k$ for some $k \geq 1$, we show that $|P| \leq (k+1) \, 2^k - 1$. 


\section{Proof of \cref{thm:acyclic}}
\label{sec:vfree}

Let us introduce two further types of restriction on posets which will be needed for the following proof, see \cref{fig:acyclic_free} for an illustration.

\begin{figure}[h]
\begin{center}
%
	\begin{tikzpicture}
	\begin{scope}
		\draw[black, line width=0.5mm] (-1,1) -- (0,0);
		\draw[black, line width=0.5mm] (1,1) -- (0,0);
		\node[circle,fill=black, inner sep=1pt, minimum size=5pt,label=above:{$p_2$}] (char) at (-1,1) {};
		\node[circle,fill=black, inner sep=1pt, minimum size=5pt,label=above:{$p_3$}] (char) at (1,1) {};
		\node[circle,fill=black, inner sep=1pt, minimum size=5pt,label=below:{$p_1$}] (char) at (0,0) {};
	\end{scope}
	\end{tikzpicture}
%
\hspace{1.5cm}
%
	\begin{tikzpicture}
	\begin{scope}
		\draw[black, line width=0.5mm] (-2,0) -- (-1,1);
		\draw[black, line width=0.5mm] (-1,1) -- (0,0);
		\draw[black, line width=0.5mm] (1,1) -- (0,0);
		\node[circle,fill=black, inner sep=1pt, minimum size=5pt,label=below:{$p_3$}] (char) at (-2,0) {};
		\node[circle,fill=black, inner sep=1pt, minimum size=5pt,label=above:{$p_1$}] (char) at (-1,1) {};
		\node[circle,fill=black, inner sep=1pt, minimum size=5pt,label=above:{$p_2$}] (char) at (1,1) {};
		\node[circle,fill=black, inner sep=1pt, minimum size=5pt,label=below:{$p_4$}] (char) at (0,0) {};
	\end{scope}
	\end{tikzpicture}
	
%
\end{center}
\caption{The picture on the left shows the forbidden type of structure which cannot occur in the Hasse diagram of a $V$-free poset.  The picture on the right shows the forbidden type of structure which cannot occur in the Hasse diagram of an $N$-free poset.}  \label{fig:acyclic_free}
\end{figure}
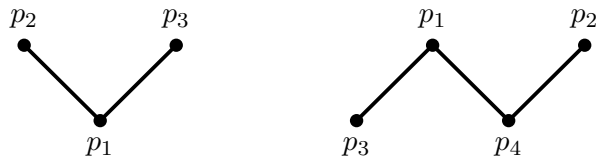

\begin{definition}
	A poset $P$ is \emph{$V$-free}, if there are no $p_1,p_2,p_3 \in P$ such that $p_1 < p_2,p_3$ and $p_2 \nsim p_3$. If there are $(p_1,p_2,p_3)$ violating that condition, we say that they form a \emph{$V$-shape}. 
	
	A poset $P$ is \emph{$N$-free}, if there are no $p_1,p_2,p_3,p_4 \in P$ such that $p_1 > p_3$,  $p_1 > p_4$, $p_2 > p_4$,  $p_1 \nsim p_2$, $p_2 \nsim  p_3$ and $p_3 \nsim p_4$. If there are $(p_1,p_2,p_3,p_4)$ violating that condition, we say that they form an \emph{$N$-shape}.
\end{definition}

Clearly a $V$-free poset necessarily is also acyclic and $N$-free. The central idea of the proof will be to go from acyclic posets first to $N$-free posets and then to $V$-free posets through a sequence of auxiliary statements. 
The following result establishes the first part of that argument. We say that a poset $P$ is \emph{connected} if its cover graph is connected.

\begin{proposition}
	For any $n$ there is a connected $N$-free poset $P \in \mT_n$ for which $\ao(P) = \ao(\mT_n)$.
\end{proposition}

\begin{proof} 
	Let $P \in \mathcal{T}_n$ such that $\ao(P) = \ao(\mT_n)$. We may assume, without loss of generality, that $P$ is connected since we can connect two disjoint components $P_1,P_2 \subset P$ by choosing a minimum element $p' \in P_1$ and a maximum element $p'' \in P_2$ and adding the relation $p' > p''$ to the Hasse diagram of $P$. This modification does not change the number of elements and does not increase the parameter $\ao (P)$, therefore the parameter does not change, since  $\ao(P) = \ao(\mT_n)$. Without loss of generality, we may also assume that $P$ is chosen such that it is connected and minimises the number of $p_2$ and $p_3$ for which there exists $p_1$, $p_2$ such that $(p_1,p_2,p_3,p_4)$ are an $N$-shape. We will prove that this number is in fact zero.

	\begin{figure}[h]
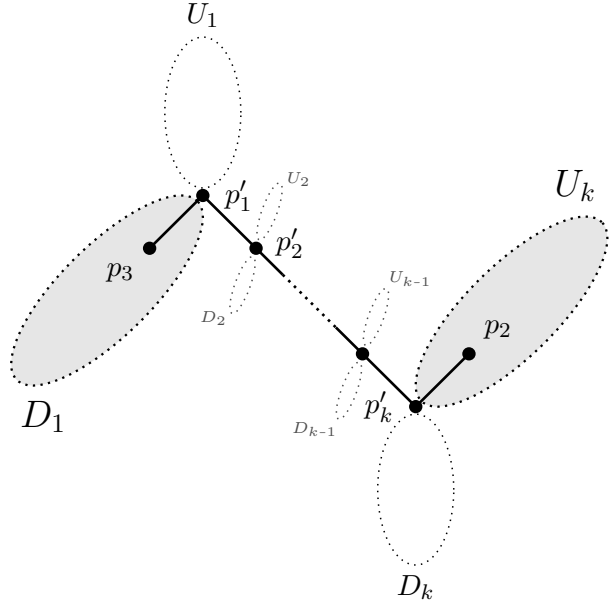

		\begin{center}
			\include{tikz_k3}
		\end{center}
	\caption{The $N$-shape $(p_1,p_2,p_3,p_4)$ in $P$ with the sets $U_i$ and $D_i$ indicated.}  \label{fig:k3}
	\end{figure}

	Assume to the contrary that there exists an $N$-shape  $(p_1,p_2,p_3,p_4)$ in $P$ where, without loss of generality, $p_1$ covers $p_3$ 
	 and $p_2$ covers $p_4$. Furthermore, let $p_1 = p_1' > p_2' > \dots > p_k' = p_4$ be a maximal chain connecting $p_1$ and $p_4$ in $P$, that is $p_{i-1}$ covers $p_i$ for any $2 \leq i \leq k$ and for some $k \geq 2$. For any $1 \leq i \leq k$, we define the following two sets of elements in $P$:
	\begin{itemize}  \setlength\itemsep{0em}
		\item[(i)] 	$U_i$ is the set of elements, such that path connecting them to $p_i'$ in the cover graph contains an element covering $p_i'$, distinct from $p_{i-1}'$.
		\item[(ii)]	$D_i$ is the set of elements, such that  path connecting them to $p_i'$ in the cover graph contains an element being covered by $p_i'$, distinct from $p_{i+1}'$.
	\end{itemize}
	Note that there is always a unique path connecting any two elements in the cover graph since $P$ is acyclic and connected. For that same reason, the sets $U_i$ and $D_i$ are completely disjoint and together with the chain $\{p_1',\dots,p_k'\}$ make up a partition of $P$. In the Hasse diagram, $U_i$ and $D_i$ are the set of elements respectively going upwards and downwards from $p_i'$ when not following the chain $\{p_1', \dots, p_k'\}$, see \cref{fig:k3} for an illustration. Note that $p_3 \in D_1$ and $p_2 \in U_k$, but otherwise we may have $U_i,D_i = \emptyset$. In fact, we may assume, without loss of generality, that 	
	\begin{equation} \label{eq:UiDi_assumption}
		U_i = D_i = \emptyset \quad \text{ for any } 2 \leq i \leq k-1,
	\end{equation}
	since otherwise we could modify our choice of $p_3$ or $p_2$ to get a pair satisfying this condition. We can likewise assume that $D_1$ is $V$-free, as otherwise we could again modify our choice of $p_3$ or $p_2$ and, if necessary, invert all relations in the poset without changing $\ao(P)$ and number of $N$-shapes. Next, we let
	\begin{equation} \label{eq:x}
		x = \max \big( k - h(D_1), 1 \big)
	\end{equation}
	and note that $x < k$ since $h(D_1) \geq 1$ and $k \geq 2$.

	\begin{figure}[h]
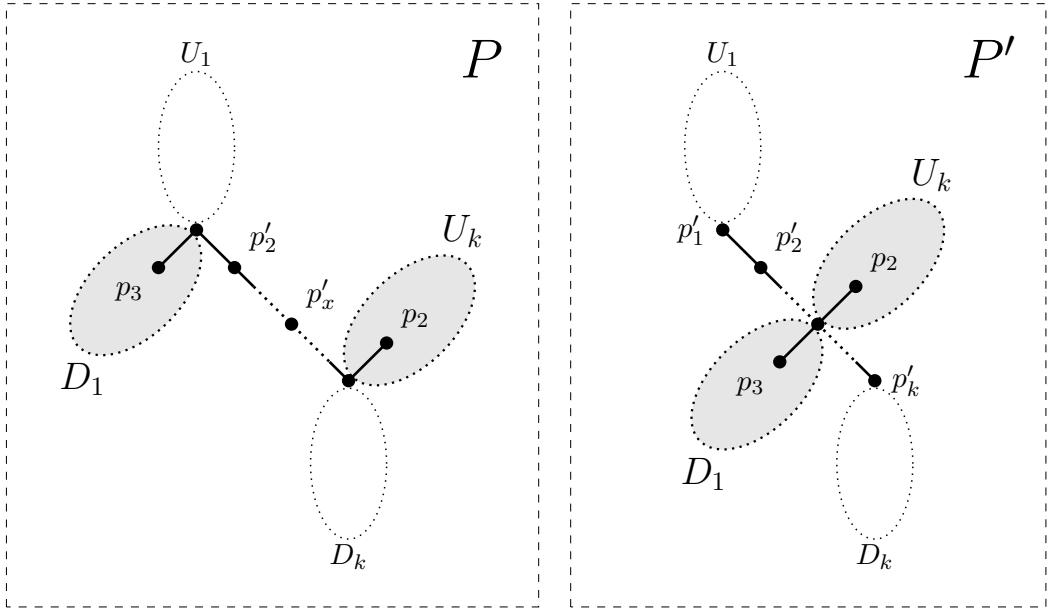

	\bigskip
		\begin{center}
			\begin{minipage}{0.4\textwidth}
				\include{tikz_k3_P_decolorized}
			\end{minipage}
			\hspace{0.05\textwidth}
			\begin{minipage}{0.4\textwidth}
				\include{tikz_k3_Pd_decolorized}
			\end{minipage}
		\end{center}
	\caption{The poset $P$ and its modification $P'$.}  \label{fig:kmod}
	\end{figure}

	We now define $P'$ to be the poset obtained from $P$ by making the following changes in the Hasse diagram: each edge from $p_1'$ to an element in $D_1$ is replaced with an edge from $p_x'$ to that same element, moving $D_1$ downward appropriately, and each edge from an element in $U_k$ to $p_k'$ is replaced with an edge from that same element to $p_x'$, moving $U_k$ upward appropriately. This just means that we are 'shifting' $D_1$ downwards and $U_k$ upwards so that both sets are connected to $p_x' \in \{p_1',\ldots,p_{k-1}'\}$. See \cref{fig:kmod} for an illustration of how $P'$ is obtained from $P$.
	
	Clearly the size of the poset remains unchanged, that is $|P'| = \abs{P}$. We also note that  $P'$ contains  fewer $N$-shapes than $P$, since now $p_2 > p_3$, so any $N$-shapes containing $p_2$ and $p_3$ have been removed, and no additional ones have been added due to \cref{eq:UiDi_assumption}. Let us therefore show that $\ao(P') \leq \ao(P)$, giving us a contradiction to our assumption that $P$ maximises $\ao(P)$ and minimises the number of $p_2,p_3$ in an $N$-structure
	
	Let $A \subset P'$ be a disjoint collection of chains of size $|A| = \ao(P')$ and let us show that we can find a disjoint collection of chains of equal size in $P$. We will do a case distinction based on whether $p_1'$, $p_x'$ and $p_k'$ are in $A$ or not. Note that by maximality of $A$, we cannot have $p_1',p_k' \in A$ but $p_x' \notin A$.

	\begin{figure}[h]
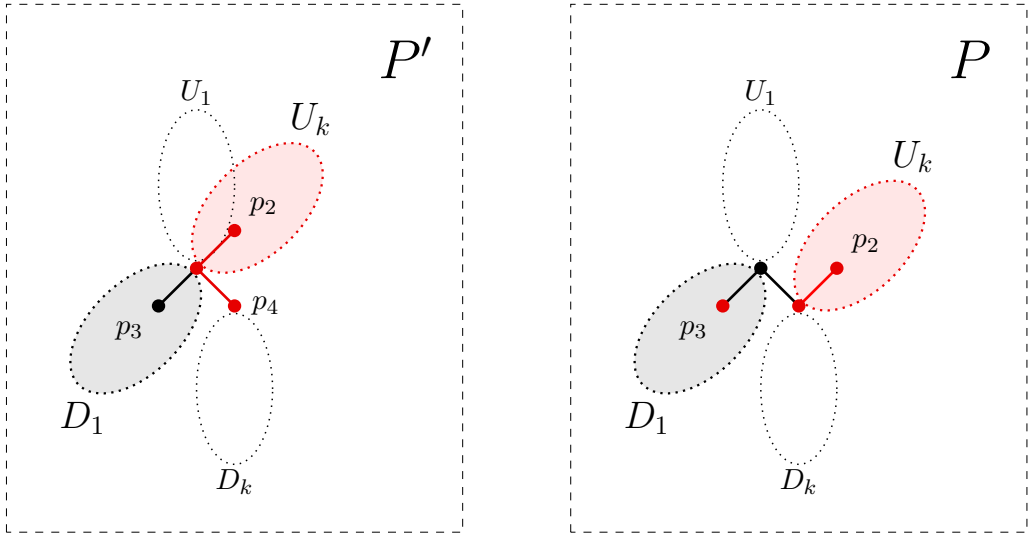

	\bigskip
		\begin{center}
			\begin{minipage}{0.4\textwidth}
				\include{tikz_k2_Pd}
			\end{minipage}
			\hspace{0.05\textwidth}
			\begin{minipage}{0.4\textwidth}
				\include{tikz_k2_P}
			\end{minipage}
		\end{center}
	\caption{The poset $P$ and its modification $P'$ illustrated for the case $k = 2$. The modification of $A$ to $A'$ in Case 5 is highlighted in red. Note that the modified version of the poset $P$ is $P'$, which is shown on the left, but the modified version of the set $A \subseteq P'$ is $A' \subseteq P$, which is shown on the right.}  \label{fig:k2P}
	\end{figure}

	\medskip
	\noindent {\bf Case 1.} If $p_1',p_x',p_k' \notin A$, then by maximality of $A$ we must have $A \cap \{p_1', \dots, p_k'\} = \emptyset$ and therefore $A$ is also a disjoint collection of chains in $P$.
	
	\medskip
	\noindent {\bf Case 2.} If $p_1' \in A $ but $p_x',p_k' \notin A$, then by maximality of $A$ we must have $A \cap \{p_1', \dots, p_k'\} = \{p_1', \dots, p_{x-1}'\}$. Also by maximality, we know that $p_1',\dots,p_{x-1}'$ are not part of the same chain as any point of $D_1$ in $P'$, as  $p_x' \notin A$. It follows that the same set $A$ is also a disjoint collection of chains in the poset $P$.
	
	\medskip
	\noindent {\bf Case 3.} If $p_x' \in A$ but $p_1',p_k' \notin A$, then by maximality of $A$ we must have $A \cap \{p_1', \dots, p_k'\} = \{p_{x}'\}$. While $p_x'$ in $A$ might be part of a chain that contains elements from $D_1$ and $U_k$, the set $A$ is still also a disjoint collection of chains in $P$. While the number of chains that $A$ consists of, when considered in $P$, may have increased, clearly the size of the set remains unchanged.
	
	\medskip
	\noindent {\bf Case 4.} If $p_1',p_x' \notin A$ but $p_k' \in A$, then by maximality of $A$ we must have $A \cap \{p_1', \dots, p_k'\} = \{p_{x+1}', \dots, p_k'\}$. As in Case 2, we know that $p_{x+1}',\dots,p_k'$ cannot be part of the same chain as any point of $U_k$, in $P'$, as otherwise we would have $p_x' \in A$. It follows that $A$ is again  a disjoint collection of chains in the poset $P$.
	
	\medskip
	\noindent {\bf Case 5.} If $p_1' \notin A$ but $p_x',p_k' \in A$, then by maximality of $A$ we must have $A \cap \{p_1', \dots, p_k'\} = \{p_{x}', \dots, p_k'\}$. Since $D_1$ is $V$-free and $x < k$, we also note that $A \cap D_1 = \emptyset$. If $ p_{x}', \dots, p_k'$ are not part of the same chain as any point in $U_k$, then $A$ is again a disjoint collection of chains in $P$. If however the chain to which $p_{x}', \dots, p_k'$ belong extends into $U_k$, then let $q_1, q_2, \dots , q_{h(D_1)}$ be a chain of length $h(D_1)$ in $D_1$ and observe that $A' = \big( A \setminus \{p_x',\dots, p_{k-1}'\} \big) \cup \{ q_1, q_2, \dots, q_{h(D_1)} \}$ is again a disjoint collection of chains in $P$ that is of size $|A'| = |A| - (k - x) + h(D_1) \geq |A|$ by \cref{eq:x}. See \cref{fig:k2P} for an illustration of this case when $k = 2$ and $h(D_1) = 1$.

	\begin{figure}[h]
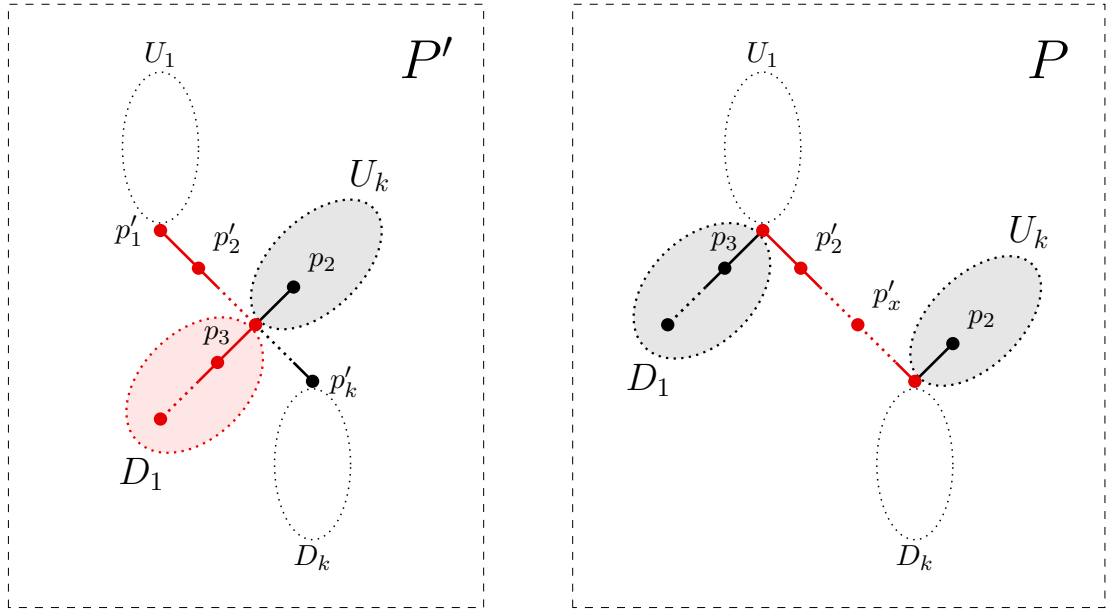

	\bigskip
		\begin{center}
			\begin{minipage}{0.4\textwidth}
				\include{tikz_k3_Pd}
			\end{minipage}
			\hspace{0.05\textwidth}
			\begin{minipage}{0.4\textwidth}
				\include{tikz_k3_P}
			\end{minipage}
		\end{center}
	\caption{The poset $P$ and its modification $P'$ illustrated for general $k$. The modification of $A$ to $A'$ in Case 6 is highlighted in red.}  \label{fig:k3P}
	\end{figure}

	\medskip
	\noindent {\bf Case 6.} If $p_1',p_x' \in A$ but $p_k' \notin A$, then by maximality of $A$ we must have $A \cap \{p_1', \dots, p_k'\} = \{p_1', \dots, p_x'\}$. If $p_1', \dots, p_x'$ are not part of the same chain as a point in $D_1$, then $A$ is again a disjoint collection of chains in $P$. If however the chain $\mathcal{C}$ to which $p_1', \dots, p_x'$ belong extends into $D_1$, then we must have $| \mathcal{C} \cap D_1 | = h(D_1)$ since $D_1$ is assumed to be $V$-free and $A$ is maximal. It follows that in $P$ we can extend the chain down to $p_k'$ instead of into $D_1$ without decreasing the size of the collection, that is $A' = \big( A \setminus ( \mathcal{C} \cap D_1 ) \big) \cup \{p_{x+1}', \dots, p_k'\}$ is a disjoint collection of chains in $P$ that is of size $|A'| = |A| - h(D_1) + k - x \geq |A|$ by \cref{eq:x}. See \cref{fig:k3P} for an illustration of this case.
	
	\medskip
	\noindent {\bf Case 7.} If $p_1',p_x',p_k' \in A$, then by maximality of $A$ we must have $\{p_1', \dots, p_x'\} \subset A$ and the chain containing $\{p_1', \dots, p_k'\}$ cannot contain any element of $U_k$ or $D_1$. It follows that $A$ is also a disjoint collection of chains in $P$.
	
	\medskip
	
	We have therefore shown that $P$ must be a connected and $N$-free poset $\mathcal{T}_n$ satisfying ${\ao(P)} = \ao(\mathcal{T}_n)$, concluding the proof.
\end{proof}

The next lemma shows that a connected, acyclic and $N$-free poset simply decomposes into two $V$-free posets, one of which is inverted, which are connected through a central element.

\begin{lemma} \label{lemma:central_element}
	For any connected, acyclic and $N$-free poset $P$ there exists some element $p' \in P$ such that any other element is comparable to it.
\end{lemma}

\begin{proof} 
	Let $p_0$ be a minimal element in $P$. Let $p_0, p_1, \dots, p_k$ be a sequence of elements in $P$ for some $k \geq 0$ such that $p_{i+1}$ is the lone element covering $p_i$ for any $0 \leq i < k$ and $p_k$ is the first element in this sequence either not covered at all or covered by at least two distinct elements. Let us show that $p' = p_k$ fulfils the desired property. Assume to the contrary that there is some $p \nsim p_k$. Since $P$ is connected and acyclic, $p_k$ and $p$ share either a common smaller or a common greater element $q$.
	
	\medskip
	\noindent \textbf{Case 1.} If $q$ is a common smaller element of $p_k$ and $p$, then $k > 0$ since otherwise we get a contradiction to the assumption that $p_0$ is a minimal element in $P$. We note that $q \nsim p_0$ since our poset is acyclic. It follows that $(p_k, p, p_0, q)$ form an $N$-shape, a contradiction.
	
	\medskip
	\noindent \textbf{Case 2.} If $q$ is a common greater element of $p_k$ and $p$, then $p_k$ has at least two distinct elements covering it, one of which must be incomparable to $p$ since our poset is acyclic. Denote this element by $q'$ and note that $q' \nsim q$ since both cover $p_k$. It follows that $(q',q,p_k,p)$ also form an $N$-shape, again a contradiction.
\end{proof}

Let us introduce some final bit of notation.

\begin{definition}
	For any $a \geq 0$ and $h \geq 1$, we write
	\begin{align*}
		\Lambda(a,h) & = \max \big\{ \abs{P} : P \text{ is a }  V\text{-free poset with } \ao(P) = a \text{ and } h(P) \leq h \big\} \text{ and} \\
		X(a) & = \max \big\{ \abs{P} : P \text{ is an acyclic and } N\text{-free poset with } \ao(P) = a \big\}.
	\end{align*}
	When $h = a$, we simply write $\Lambda(a) = \Lambda(a,a)$.
\end{definition}


Note that if $h > a$, then clearly $\Lambda(a,h) = \Lambda(a,a) = \Lambda(a)$. Our goal will ultimately be to determine $X(a)$ in terms of $\Lambda(a)$ and to find a closed expression for both. The following three statements first establish how to determine $\Lambda(a)$.

\begin{proposition} \label{prop:lambda}
	For any $a \geq 2$ we have 
	\begin{equation} \label{eq:lambda_max}
		\Lambda(a) =  \max \big\{ \Lambda (f) + \Lambda (a-f)  + a-f : a/2 \leq f < a \big\}.
	\end{equation}
	%
\end{proposition}
\begin{proof} 

	\medskip
	\noindent \textbf{A lower bound} follows by considering the following construction: given any $a/2 \leq f < a$, take a $V$-free poset $P_1$ satisfying $\abs{P_1} = \Lambda (f)$ and $\ao(P_1) = f$ as well as a $V$-free poset $P_2$ satisfying $\abs{P_2} = \Lambda (a-f)$ and $\ao(P_2) = a-f$. It is easy to see that the poset $P$ obtained by taking the disjoint union of $P_1$ and $P_2$ and adding $k = a-f$ elements $p_1,\dots,p_{k}$ such that $p_{k} > \dots > p_2 > p_1$ as well as $p_1 > p$ for all $p \in P_1,P_2$ satisfies $\abs{P} = a-f+ \Lambda (f) + \Lambda (a-f)$, $h(P)=a$ and $\ao(P) = \ao(P_1) + \ao(P_2) = a$ as desired. 
	
	\medskip
	
	\noindent For \textbf{a matching upper bound}, let $P$ be a poset satisfying $\abs{P} = \Lambda(a)$ and $\ao(P) = a$. We must also have $h(P)=a$, as otherwise one could add an additional element to $P$ that is greater than every other element, which increases the size of $P$ without increasing $\ao(P)$. This also establishes that $P$ is connected, as the sum of the heights of two disjoint parts would need to be at most $a$.

	Consider therefore a chain $p_1 > p_2 > \dots > p_a$ of length $a$ in $P$. Note that, by the previous construction, $\abs{P} > a$ since $a \geq 2$. Let $p_x$ denote the largest element in that chain which covers at least two distinct elements in $P$ and assume that $P$ was chosen such that $x$ (the integer, not the poset element $p_x$) is minimal. By removing $p_1, \dots, p_x$ from $P$ we obtain a poset consisting of at least two incomparable parts. Let $P_1$ be the part containing the rest of the chain and $P_2$ the rest of $P$ without the initial $x$ elements. The height of $P_1$ is $a-x$ and therefore $\ao(P_1) \geq a-x$. If $\ao(P_1) > a-x$ then trivially $x>1$ and by moving $p_1$ (the largest element of $P$) to be in between $p_x$ and the maximal elements of $P_1$ in the Hasse diagram of $P$, we obtain a contradiction to our assumption that $x$ is minimal.  It follows that $\ao(P_1) = a-x$. Since $P_1$ and $P_2$ are incomparable, we have $\ao(P_1) + \ao(P_2) \leq a$ which implies that $\ao(P_2) \leq x$ and therefore 
	\begin{equation*}
		\abs{P} = x + \abs{P_1} + \abs{P_2} \leq x+ \Lambda(a-x) + \Lambda(x).
	\end{equation*}
	This establishes the statement of the proposition.
\end{proof}

Let us show that the formula in \cref{prop:lambda} is maximised by $f = \ceil{a/2}$.

\begin{lemma}\label{half}
	For any $a \geq 2$, we have
	\begin{equation*}
		\Lambda(a) = \Lambda(\ceil{a/2}) + \Lambda (\floor{a/2})  + \floor{a/2}.
	\end{equation*}
\end{lemma}
\begin{proof}
	The proof follows by induction on $a$. The base case $a=2$ is easy to see. Let us therefore show that the statement holds for $a$ assuming it has been proven for all smaller values.
	
	By inductive assumption, we have
	\begin{align*}
		\Lambda(f) &= \Lambda(\ceil{f/2}) + \Lambda (\floor{f/2})  + \floor{f/2} \text{ and} \\
		\Lambda (a-f) &= \Lambda(\ceil{(a-f)/2}) + \Lambda (\floor{(a-f)/2})  + \floor{(a-f)/2}.
	\end{align*}
	Note that this holds even if $f = 1$ or $a-f = 1$ since $\Lambda(0) = 0$. Let us distinguish two cases based on the parity of $f$ and $a-f$.
	
	\medskip
	\noindent{\bf Case 1.} If at least one of $f$ and $a-f$ is even, we have that $\ceil{(a-f)/2} + \ceil{f/2} = \ceil{a/2}$ as well as $\floor{(a-f)/2} + \floor{f/2} = \floor{a/2}$. Using the inductive hypothesis for $\ceil{a/2}$ and $\floor{a/2}$, it therefore follows that
	\begin{align*}
		& \quad \Lambda(f) + \Lambda (a-f)  + a-f \\
		& \leq \Lambda ( \ceil{a/2} ) - \ceil{(a-f)/2} + \Lambda ( \floor{a/2} ) - \floor{(a-f)/2}  + \floor{f/2} + \floor{(a-f)/2} + (a-f) \\
		& = \Lambda ( \ceil{a/2} ) + \Lambda ( \floor{a/2} ) + \floor{a/2}.
	\end{align*}
	Here we have used \cref{prop:lambda} and that $\ceil{(a-f)/2} + \floor{(a-f)/2} = a-f$ and $\floor{f/2} + \floor{(a-f)/2} = \floor{a/2}$.
	
	\medskip
	\noindent{\bf Case 2.} If both $f$ and $a-f$ are odd, so that $a$ must be even, then $\ceil{f/2} + \floor{(a-f)/2} = \floor{f/2} + \ceil{(a-f)/2} = a/2$. Also note that $\ceil{(a-f)/2} \leq \floor{f/2}$ as $f > a-f$	. Using the inductive hypothesis for $a/2$, it therefore again  follows that
	\begin{align*}
		& \quad \Lambda(f) + \Lambda (a-f)  + a-f \\
		& \leq 2 \, \Lambda ( a/2 ) - \ceil{(a-f)/2} - \floor{(a-f)/2}  + (\floor{a/2} - 1) + (a-f) \\
		& < \Lambda ( \ceil{a/2} ) + \Lambda ( \floor{a/2} ) + \floor{a/2}.
	\end{align*}
	Here we have again used that \cref{prop:lambda} and that $\floor{f/2} + \floor{(a-f)/2} = \floor{a/2} - 1$.
	%
%
%
\end{proof}

Let us now show that the formula in \cref{prop:lambda} is also maximised by $f = 2^{\ceil{\log_2 a} - 1}$.

\begin{proposition}\label{2power}
	For any $a \geq 2$, we have
	\begin{equation*}
		\Lambda(a) = \Lambda(2^{\ceil{\log_2 a}-1}) + \Lambda (a-2^{\ceil{\log_2 a}-1})  + (a-2^{\ceil{\log_2 a}-1}).
	\end{equation*}
	%
\end{proposition}

\begin{proof}
    The proof again follows by induction on $a$, where the case $a=2$ is easy to see. Let $\alpha = \ceil{\log_2 a} -1$, so that $2^{\alpha} <  a \leq 2^{\alpha+1}$. We may in fact assume that $2^{\alpha} <  a < 2^{\alpha+1}-1$, as otherwise $2^{\alpha-1} = \ceil{a/2}$, so that the result otherwise follows by \cref{half}. It follows that $2^{\alpha-1} \leq \floor{a/2} \leq \ceil{a/2} < 2^{\alpha}$, so that by inductive assumption 
    %
    and by \cref{half}
    \begin{align*}
    \Lambda(a) & = \Lambda\left(\ceil{a/2}\right) + \Lambda\left(\floor{a/2}\right) + \floor{a/2} \\
     & = \parr{2^{\alpha-1}+ 2\Lambda\left(2^{\alpha-1}\right)} +  \parr{\floor{\frac{a}{2}}-2^{\alpha-1} + \Lambda\left(\floor{\frac{a}{2}}-2^{\alpha-1}\right) + \Lambda\left(\ceil{\frac{a}{2}}-2^{\alpha-1}\right)} + \parr{a- 2^{\alpha-1}} \\
    & =  \Lambda\parr{2^{\alpha}} + \Lambda\parr{a- 2^{\alpha}} + (a- 2^{\alpha}),
    \end{align*}
    where the last equality follows from \cref{half}.
\end{proof}


The following corollary establishes an exact formula for $\Lambda(a)$ based on the binary representation of $a$.

\begin{corollary} \label{cor:L}
	Let $a \geq 1$. If $a = 2^{i_0} + 2^{i_1} + \cdots + 2^{i_{t-1}}$ for some $i_0 < i_1 < \cdots < i_{t-1}$, then
	\begin{equation}
		\Lambda(a) = \sum_{k=0}^{t-1} (2t-2k +i_k) \, 2^{i_k-1} 
	\end{equation}
	In particular, if $a = 2^k$ for some $k \geq 1$, then $\Lambda(2^k) = 2^{k-1} \, (k+2)$.
\end{corollary}

\begin{proof} 
	This follows through an easy induction on $a$, noting that it holds for $a = 1$. If $a$ is a power of two, that is $t = 1$ and $i_0 > 0$, then by \cref{half} 
	\begin{align*}
		\Lambda(a) & = 2\big( 2^{{i_0}-2} \, (i_0+1) \big) +  2^{i_0-1} = 2^{i_0-1} \, (i_0+2). 
	\end{align*}
    If $a$ is not a power of two, then by \cref{2power}
	\begin{align*}
		\Lambda(a) & =  2^{i_{t-1}-1} (i_{t-1} + 2) + \left( \sum_{k=0}^{t-2} (2(t-1)-2k +i_k) \, 2^{i_k-1} \right) + \sum_{k=0}^{t-2} 2^{i_k} = \sum_{k=0}^{t-1} (2t-2k +i_k) \, 2^{i_k-1}.
	\end{align*}
	establishing the corollary.
\end{proof}

\begin{proposition}\label{one}
	For any $a = 2^k$ where $k \geq 1$ and $a/2 < f < a$, we have
	\begin{equation*}
		\Lambda(f) + \Lambda (a-f)  + a-f < \Lambda(a)
	\end{equation*}
\end{proposition}

\begin{proof}
	We will prove the result by induction, noting that it vacantly holds for $k=1$. Following the same argument as in Case~2 in the proof of \cref{half}, we can assume that $f$ and $a-f$ are both even. We therefore have that $\Lambda(f) = 2 \, \Lambda(f/2) + f/2$ and $\Lambda(a-f) = 2 \, \Lambda((a-f)/2) + (a-f)/2$, so that
	\begin{align*}
		\Lambda(f) + \Lambda (a-f)  + a-f  & = 2 \, \Lambda(f/2) + f/2 +2 \, \Lambda((a-f)/2) + (a-f)/2 + a- f \\
		& = 2 \big( \Lambda(f/2) + \Lambda((a-f)/2) + (a-f)/2 \big) + a/2 \\
		& < 2 \Lambda(a/2) + a/2 = \Lambda(a).
	\end{align*}
	In the last step we have used the inductive assumption since $a/2 = 2^{k-1}$.
\end{proof}

Let us now turn our attention to the function $X(a)$.

\begin{lemma} \label{eq:X_max}
	For $a \geq 1$ we have
	\begin{equation*} 
		X(a) = \max \{ \Lambda(a,h) + \Lambda(a,a-1-h) + 1 : 0 \leq h \leq (a-1)/2 \}
	\end{equation*}
\end{lemma}

\begin{proof} 
	We prove the statement through a matching upper and lower bound.

	\medskip
	\noindent \textbf{For a lower bound}, fix $0 \leq h \leq (a+1)/2$ maximising \cref{eq:X_max} and let $P_1$ be a $V$-free poset of size $\Lambda(a,h)$ with $\ao(P_1) \leq a$ and $h(P_1) \leq h$ and $P_2$ a $V$-free poset of size $\Lambda(a,a+1-h)$ with $\ao(P_2) \leq a$ and $h(P_2) \leq a-1-h$. Let $P_2'$ denote poset obtained from $P_2$ by inverting all of its relations and let $P$ be the poset obtained by joining $P_1$ and $P_2'$ through a central element that is greater than all elements in $P_1$ and smaller than all elements in $P_2'$. Clearly $h(P) = h(P_1) + h(P_2) - 1 \leq a$ and $\ao(P) \leq a$ but $\abs{P} = \Lambda(a,h) + \Lambda(a,a+1-h) - 1$ as desired.
	
	\medskip
	\noindent \textbf{For an upper bound}, let $P$ be a connected, acyclic and $N$-free poset with $\abs{P} = X\big(a\big)$ and $\ao(P) = a$. Let $p' \in P$ be an element comparable to every  element of the poset $P$. The existence of such an element is guaranteed by \cref{lemma:central_element}. We have that $h(P) = a$ as otherwise we can replace $p'$ with an edge, without increasing the parameter $\ao(P)$, a contradiction. If we remove $p'$ and its connections from the Hasse diagram of $P$, we obtain a $V$-free poset $P_1$ and as well as a poset $P_2$ whose inverse is also $V$-free. Clearly $h(P_1) + h(P_2) + 1 \leq h(P) = a$ and $\ao(P_1), \ao(P_2) \leq a$. Writing $h = \min \big( h(P_1), h(P_2) \big) \leq (h-1)/2 = (a-1)/2$, we clearly have $X(a) \leq \Lambda(a,h) + \Lambda(a,a-1-h) + 1$.
\end{proof}

\begin{proposition}\label{hlower}
	For $a \geq 2$ and $0 \leq h \leq a$, we have that
	\begin{equation}
		\Lambda(a,h) \leq \Lambda(a) - (a-h)
	\end{equation}
	If $a/2 \leq h \leq a$, then we in fact have equality.
\end{proposition}

\begin{proof}
	 A lower bound of this when $a/2 \leq h \leq a$ follows immediately from the construction in the lower bound in the proof of \cref{prop:lambda}. For an upper bound for any $0 \leq h \leq a$, let $P$ be a poset of size $\Lambda(a,h)$, height $h(P) = h$ and satisfying $\ao(P) = a$. Consider the poset $P'$ obtained by adding a chain $x_1 < x_2 < \cdots < x_{a-h}$ such that $x_1 > x$ for any $x \in P$. It is clear that $\ao(P') = a$ and therefore $\Lambda(a) \geq \abs{P'} = \abs{P} + (a-h) = \Lambda(a,h) + (a-h)$.
\end{proof}

\begin{proposition} \label{proposition_minus_one}
	Let $a \geq 2$. If $a$ is even but not a power of two, then $\Lambda(a,a/2-1)=\Lambda(a) - a/2-1$. If $a$ is a power of two, then $\Lambda(a,a/2-1)=\Lambda(a) - a/2-2$.
\end{proposition}

\begin{proof}
	Let us write $\alpha = \ceil{\log_2(a)} - 1$, that is $2^\alpha < a \leq 2^{\alpha+1}$. By \cref{hlower} we know that $\Lambda(a,a/2-1) \leq \Lambda(a)-a/2-1$.

	\medskip
	\noindent {\bf Case 1.}	If $a$ is even but not a power of two, the by \cref{2power} we have that
	\begin{equation} \label{eq:La}
		\Lambda(a) = \Lambda(2^\alpha)+\Lambda(a-2^\alpha)+(a-2^\alpha).
	\end{equation}
	Since $a/2 > 2^{\alpha-1}$, \cref{hlower} states that
	\begin{equation} \label{eq:a/2m1-1}
		\Lambda(2^\alpha,a/2-1)=\Lambda(2^\alpha)-2^{\alpha}+a/2-1.
	\end{equation}
	Since $a$ is not a power of two, we have $a/2 - 1 \geq a-2^\alpha$ and therefore
	\begin{equation} \label{eq:a/2m1-2}
		\Lambda(a-2^\alpha,a/2-1)=\Lambda(a-2^\alpha).
	\end{equation}
	Using \cref{eq:a/2m1-1} and \cref{eq:a/2m1-2}, we can choose $V$-free posets $P_1$, $P_2$  such that $\ao(P_1) = 2^\alpha$, $h(P_1) \leq a/2-1$ and $\abs{P_1} = \Lambda(2^\alpha,a/2-1)$ as well as $\ao(P_2) = a-2^\alpha$ and $\abs{P_2} = \Lambda(a-2^\alpha)$. The disjoint union of $P_1$ and $P_2$ is a $V$-free poset satisfying $\ao(P) = a$, $h(P) = a/2-1$ and has size $\Lambda(2^\alpha)-2^{\alpha}+a/2-1 + \Lambda(a-2^\alpha) = \Lambda(a) - a/2-1$ by \cref{eq:La}, establishing the result for this case.

	\medskip
	\noindent {\bf Case 2.} If $a = 2^{\alpha+1}$, then
	\begin{align*}
		\Lambda(2^{\alpha +1},a/2-1) \geq 2 \, \Lambda(2^{\alpha},2^{\alpha}-1) = 2 \, \Lambda(2^{\alpha}) - 2 = \Lambda(2^{\alpha +1}) - 2^{\alpha}-2,
	\end{align*}
	where we have used both \cref{hlower} and \cref{2power}. Assume now that there exists some $V$-free poset $P$ satisfying $\ao(P) = 2^{\alpha+1}$, $h(P) = 2^{\alpha-1}-1$ and $\abs{P} = \Lambda(2^{\alpha +1}) - 2^{\alpha}-1$. Let $P'$ be the poset obtained from $P$ by adding a chain of $2^{\alpha}+1$ elements $p_1,p_2,\dots,p_{2^{\alpha}+1}$ such that $p_{2^{\alpha}+1} > \cdots > p_2 > p_1$ and $p_1 > p$ for any $p\in P$. Clearly $P'$ is $V$-free and satisfies $|P'| = \Lambda(2^{\alpha}+1)$, $h(P') = 2^{\alpha} + 2^{\alpha-1} < 2^{\alpha+1}$ as well as $\ao (P') = a$.
	
	Let $y$ now denote the minimal integer such that if we iteratively delete the first $y$ maximal elements of $P'$, the resulting poset becomes disconnected. By construction it is clear that $y \geq 2^{\alpha}+1$. Let $P_1$ and $P_2$ be two disjoint (but not necessarily connected) and non-empty posets making up $P'$ after deleting the first $y$ maximal elements. Let us now without loss of generality assume that $\ao(P_1)\leq \ao(P_2)$ and write $x = \ao(P_1)$.     Since $\ao(P) = 2^{\alpha+1}$ and $y \geq 2^{\alpha}+1$, it is clear that $x \leq 2^\alpha - 1 \leq y-2$. By \cref{hlower} and \cref{prop:lambda}, it follows that
	\begin{align*}
         \Lambda(a) = \abs{P'} & \leq y+\Lambda(x,2^{\alpha +1}-y)+\Lambda(2^{\alpha +1}-x,2^{\alpha +1}-y)\\ & \leq y+\Lambda(x)+\Lambda(2^{\alpha +1}-x)-(y-x)\leq \Lambda(2^{\alpha +1}).
     \end{align*}
     We note that by \cref{hlower} equality in the second inequality only holds if $x \leq 2^{\alpha+1}-y$ and by \cref{one} equality in the second inequality only holds if $x = 2^{\alpha}$. Since $y \geq 2^{\alpha}+1$ these two cases cannot occur simultaneously, giving us the contradiction $\Lambda(a) < \Lambda(2^{\alpha +1})$.
   \end{proof}

We can now determine when the expression in \cref{eq:X_max} is maximised.

\begin{proposition} \label{prop:X_maximiser}
	For any $a \geq 2$, we have
	\begin{equation*} 
		X(a) = \Lambda(a,\floor{(a-1)/2}) + \Lambda(a,\ceil{(a-1)/2}) + 1.
	\end{equation*}
\end{proposition}

\begin{proof}
	By \cref{hlower} we know that $\Lambda(a,a-h) - \Lambda(a,a-h-1) = 1$ for any $1 \leq h \leq (a-1)/2$ since $a-h \geq (a+1)/2 \geq \ceil{a/2}$. Since $\Lambda(a,h) - \Lambda(a,h-1) \geq 1$, it follows that the expression in \cref{eq:X_max} is maximised when $h = \floor{(a-1)/2}$.
\end{proof}

\begin{corollary} \label{cor:X}
	Let $a \geq 2$. If $a = 2^{i_0} + \cdots + 2^{i_{t-1}}$ for some $i_0 < \cdots < i_{t-1}$, then
	\begin{equation}
		X(a) = \sum_{k=0}^{t-1} \big( 2(t-k) +i_k-1 \big) \, 2^{i_k} - 1 + \ceil{\log_2 a} - \floor{\log_2 a}.
	\end{equation}
	In particular, if $a = 2^k$ for some $k \geq 1$, then $X(a) = (k+1) \, 2^k - 1$.
	%
	%
\end{corollary}

\begin{proof} 
	By \cref{prop:X_maximiser} and \cref{proposition_minus_one}, we have that
	\begin{equation*}
		X(a) = 2\Lambda(a) - a - 1 + \ceil{\log_2 a} - \floor{\log_2 a}
	\end{equation*}
	so that the result follows using \cref{cor:L}.
\end{proof}

\begin{proof}[Proof of \cref{thm:acyclic}]
	Given $n$, let $k \in \mathbb{N}$ be such that $k \, 2^{k-1} - 1 < n \leq (k+1) \, 2^k - 1$. By the previous corollary it follows that $2^{k-1} < \ao(\mT_n) \leq 2^{k} $,	from which the desired asymptotic behaviour follows.
\end{proof}

\section{Remarks and Open Questions} \label{sec:remarks}

It would be of interest to further explore how the parameter $\ao(P)$ relates to natural restrictions of the poset. A commonly studied type of posets are those whose cover graph (or sometimes the Hasse diagram) are planar. A first question could be if one can improve the general lower bound of $\ceil{\sqrt{n}}$ in that case. The best current constructions of planar posets contain families of independent chains of size $\sqrt{2n}$, contrasting with the lower bound of $\sqrt{n}$. The question therefore becomes about determining the correct leading coefficient.

Another important restriction one can impose on a poset $P$ is to bound its \emph{dimension}, that is the least number of linear orderings needed to describe $P$ as their intersection. Since the 70s, several key results have been established which relate the dimension of a poset to the planarity of its Hasse diagram, cover graph or comparability graph. Most notably, posets with planar Hasse diagram have small dimension if they have a maximum or minimum element~\cite{BakerFishburnRoberts-1971, TrotterMoore-1977}, but otherwise can have large dimension~\cite{Kelly-1981}.

More recently, strong connections have been made between the dimension of posets and certain graph parameters. A Theorem of Dilworth bounds the dimension from above by the width of the poset. Remarkably, in~\cite{StreibTrotter-2014} it was shown that the dimension of a poset with a planar cover graph is also bounded by the height, proving a conjecture of Felsner, Li and Trotter~\cite{FelsnerLiTrotter-2010}. Very recently polynomial bounds were obtained for that relationship in~\cite{KozikMicekTrotter-2019} and in~\cite{JoretMicekWiechert-2017} linear bounds were obtained for the more restrictive case of posets with a planar Hasse diagram.

Given the fact that planar posets of large dimension can be neither short nor narrow and the relation between $\ao(P)$ and the width as well as the height of a poset, it would be of interest to explore if a similar relation between the  $\ao(P)$ and the dimension of planar posets exists. Here it is important that, unlike the width or height, the parameter will depend on the cardinality of the poset. Likewise, the dimension would have to grow sufficiently fast with the size of the poset. In particular, one might ask if $\min \log \ao(P) / \log n = 1/2 + o_n(1)$ always holds when minimising over all planar poset $P$ of size $n$ and dimension $d = d(n)$, or if the exponent on the right-hand side increases when $d$ grows fast enough.

Returning to the more general setting of graph classes, it is very natural to consider the case of planar graphs.  Let $\mathcal{G}_n$ be the family of all $n$ vertex planar graphs.  Since the independence number of an $n$-vertex planar graph is at least $n/4$ we have that $\ao(\mathcal{G}_n) \ge n/4$.  On the other hand, when 7 divides $n$, consider the graph $G$ obtained by taking vertex disjoint copies of the join of a $5$-cycle and an independent set of size 2.  In each connected component the value of the parameter is $3$, so this provides a construction proving that $\ao(\mathcal{G}_n) \le 3n/7$.  It would be interesting to determine more precise estimates on $\ao(\mathcal{G}_n)$.

\bigskip
\noindent \textbf{Acknowledgements.} We would like to thank Maria Axenovich, who first brought this problem to the attention of the third author for the case of $V$-free acyclic posets. The second author would also like to thank Ervin Gy\H{o}ri and the Combinatorics and Discrete Mathematics group at the Alfr\'ed R\'enyi Institute for their hospitality.

The research of the first author was partially supported by the National Research, Development and Innovation Office -- NKFIH, grant K 116769 and SNN 117879 and by the Shota Rustaveli National Science Foundation of Georgia SRNSFG, grant number DI-18-118.   
The research of the second author was supported by the Spanish Ministerio de Econom\'ia y Competitividad FPI grant under the project MTM2014-54745-P and the Mar\'ia de Maetzu research grant MDM-2014-0445.
The research of the third author is supported by IBS-R029-C1.         
         

\end{document}

%% file: tikz_k3.tex
\begin{tikzpicture}
\begin{scope}
	\coordinate (ZERO) at (0,0);
	\coordinate (R) at (0.7,0);
	\coordinate (L) at (-0.7,0);
	\coordinate (D) at (0,-0.7);
	\coordinate (U) at (0,0.7);
	\coordinate (UR) at (0.7,0.7);
	\coordinate (UL) at (-0.7,0.7);
	\coordinate (DR) at (0.7,-0.7);
	\coordinate (DL) at (-0.7,-0.7);
	\coordinate (P3) at (ZERO);
	\coordinate (P1) at ($(P3) + (UR)$);
	\coordinate (P1') at (P1);
	\coordinate (P2') at ($(P1') + (DR)$);
	\coordinate (P3') at ($(P2') + (DR)$);
	\coordinate (P4') at ($(P3') + (DR)$);
	\coordinate (P5') at ($(P4') + (DR)$);
	\coordinate (Px') at ($0.5*(P3') + 0.5*(P4')$);
	\coordinate (P4) at (P5');
	\coordinate (P2) at ($(P4) + (UR)$);
	\draw[rotate=45, black, fill=black!10, line width=0.3mm, dotted] ($(P1') + 1.8*(DL)$) ellipse (1.7*1.41*0.7 and 0.6);
	\node[label=below:{\Large $D_1$}] (char) at ($(P3)+2*(DL)+0.5*(D)$){};
	\draw[rotate=45, black, fill=black!10, line width=0.3mm, dotted] ($(P2)+0.8*(UR)$) ellipse (1.7*1.41*0.7 and 0.6);
	\node[label=above:{\Large $U_k$}] (char) at ($(P2)+2*(UR)+0.5*(U)$){};
	\draw[black, line width=0.2mm, dotted] ($(P1')+(0,1.1)$) ellipse (0.5 and 1);
	\node[label=above:{$U_1$}] (char) at ($(P1')+2.8*(U)$){};
	\draw[black, line width=0.2mm, dotted] ($(P5')-(0,1.1)$) ellipse (0.5 and 1);
	\node[label=below:{$D_k$}] (char) at ($(P5')+2.8*(D)$){};
	\draw[black!70, rotate=-20, line width=0.2mm, dotted] ($(P2')+(0,0.52)$) ellipse (0.1 and 0.4);
	\node[label=above:{\color{black!70} \tiny $U_2$}] (char) at ($(P2')+0.8*(UR)$){};
	\draw[black!70, rotate=-20, line width=0.2mm, dotted] ($(P2')-(0,0.52)$) ellipse (0.1 and 0.4);
	\node[label=below:{\color{black!70} \tiny $D_2$}] (char) at ($(P2')-0.8*(UR)$){};
	\draw[black!70, rotate=-20, line width=0.2mm, dotted] ($(P4')+(0,0.52)$) ellipse (0.1 and 0.4);
	\node[label=above:{\color{black!70} \tiny $U_{k \shortminus 1}$}] (char) at ($(P4')+0.9*(UR)$){};
	\draw[black!70, rotate=-20, line width=0.2mm, dotted] ($(P4')-(0,0.52)$) ellipse (0.1 and 0.4);
	\node[label=below:{\color{black!70} \tiny $D_{k \shortminus 1}$}] (char) at ($(P4')-0.9*(UR)$){};
	\draw[black, line width=1pt] (P1) -- (P3);
	\draw[black, line width=1pt] (P1') -- ($0.5*(P2') + 0.5*(P3')$);
	\draw[black, dotted, line width=1pt] ($0.5*(P2') + 0.5*(P3')$) -- ($0.5*(P3') + 0.5*(P4')$);
	\draw[black, line width=1pt] ($0.5*(P3') + 0.5*(P4')$) -- (P5');
	\draw[black, line width=1pt] (P4) -- (P2);
	\node[circle, fill=black, inner sep=1pt, minimum size=5pt, label=225:{$p_3$}] (char) at (P3) {};
	\node[circle, fill=black, inner sep=1pt, minimum size=5pt, label=45:{$p_2$}] (char) at (P2) {};
	\node[circle, fill=black, inner sep=1pt, minimum size=5pt, label=right:{$\,p_1'$}] (char) at (P1') {};
	\node[circle, fill=black, inner sep=1pt, minimum size=5pt, label={[xshift=1pt, yshift=2pt] right:$p_2'$}] (char) at (P2') {};
	\node[circle, fill=black, inner sep=1pt, minimum size=5pt, label=225:{}] (char) at (P4') {};
	\node[circle, fill=black, inner sep=1pt, minimum size=5pt, label=left:{$p_k'\,$}] (char) at (P5') {};
\end{scope}
\end{tikzpicture}

%% file: tikz_k3_P_decolorized.tex
\begin{tikzpicture}
\begin{scope}
	\coordinate (ZERO) at (0,0);
	\coordinate (R) at (0.5,0);
	\coordinate (L) at (-0.5,0);
	\coordinate (D) at (0,-0.5);
	\coordinate (U) at (0,0.5);
	\coordinate (UR) at (0.5,0.5);
	\coordinate (UL) at (-0.5,0.5);
	\coordinate (DR) at (0.5,-0.5);
	\coordinate (DL) at (-0.5,-0.5);
	\coordinate (P3) at (ZERO);
	\coordinate (P32) at ($(P3) + 1.5*(DL)$);
	\coordinate (P1) at ($(P3) + (UR)$);
	\coordinate (P1') at (P1);
	\coordinate (P2') at ($(P1') + (DR)$);
	\coordinate (P3') at ($(P2') + (DR)$);
	\coordinate (P4') at ($(P3') + (DR)$);
	\coordinate (P5') at ($(P4') + (DR)$);
	\coordinate (Px') at ($0.5*(P3') + 0.5*(P4')$);
	\coordinate (P4) at (P5');
	\coordinate (P2) at ($(P4) + (UR)$);
	\draw[dashed, line width=0.1mm] ($(ZERO) + 4*(L) + 9*(D)$) -- ($(ZERO) + 4*(L) + 7*(U)$) -- ($(ZERO) + 10*(R) + 7*(U)$) -- ($(ZERO) + 10*(R) + 9*(D)$) -- ($(ZERO) + 4*(L) + 9*(D)$);
	\node[label={\huge $P$}] (char) at ($(ZERO) + 8.5*(R) +4.5*(U)$) {};
	\draw[rotate=45, black, fill=black!10, line width=0.3mm, dotted] ($(P1') + 1.6*(DL)$) ellipse (1.06 and 0.6);
	\node[label=below:{\Large $D_1$}] (char) at ($(P3)+2*(DL)$){};
	\draw[rotate=45, black, fill=black!10, line width=0.3mm, dotted] ($(P2)+0.6*(UR)$) ellipse (1.06 and 0.6);
	\node[label=above:{\Large $U_k$}] (char) at ($(P2)+2*(UR)$){};
	\draw[black, line width=0.2mm, dotted] ($(P1')+(0,1.1)$) ellipse (0.5 and 1);
	\node[label=above:{$U_1$}] (char) at ($(P1')+3.8*(U)$){};
	\draw[black, line width=0.2mm, dotted] ($(P5')-(0,1.1)$) ellipse (0.5 and 1);
	\node[label=below:{$D_k$}] (char) at ($(P5')+3.8*(D)$){};

	\draw[black, line width=1pt] (P1) -- (P3);
	\draw[black!10!black, line width=1pt] (P1') -- ($0.5*(P2') + 0.5*(P3')$);
	\draw[black!10!black, dotted, line width=1pt] ($0.5*(P2') + 0.5*(P3')$) -- ($0.5*(P4') + 0.5*(P5')$);
	\draw[black!10!black, line width=1pt] ($0.5*(P4') + 0.5*(P5')$) -- (P5');
	\draw[black, line width=1pt] (P4) -- (P2);
	\node[circle, fill=black, inner sep=1pt, minimum size=5pt, label=225:{$p_3$}] (char) at (P3) {};
	\node[circle, fill=black, inner sep=1pt, minimum size=5pt, label=45:{$p_2$}] (char) at (P2) {};
	\node[circle, fill=black!10!black, inner sep=1pt, minimum size=5pt, label=above:{}] (char) at (P1') {};
	\node[circle, fill=black!10!black, inner sep=1pt, minimum size=5pt, label=45:{$p_2'$}] (char) at (P2') {};
	\node[circle, fill=black!10!black, inner sep=1pt, minimum size=5pt, label=45:{$p_x'$}] (char) at (Px') {};
	\node[circle, fill=black!10!black, inner sep=1pt, minimum size=5pt, label=below:{}] (char) at (P5') {};
\end{scope}
\end{tikzpicture}

		

%% file: tikz_k3_Pd_decolorized.tex
\begin{tikzpicture}
\begin{scope}
	\coordinate (ZERO) at (0,0);
	\coordinate (R) at (0.5,0);
	\coordinate (L) at (-0.5,0);
	\coordinate (D) at (0,-0.5);
	\coordinate (U) at (0,0.5);
	\coordinate (UR) at (0.5,0.5);
	\coordinate (UL) at (-0.5,0.5);
	\coordinate (DR) at (0.5,-0.5);
	\coordinate (DL) at (-0.5,-0.5);
	\coordinate (P3) at ($(ZERO) + (L)$);
	\coordinate (P32) at ($(P3) + 1.5*(DL)$);
	\coordinate (P1) at ($(P3) + (UR)$);
	\coordinate (P1') at (P1);
	\coordinate (P2') at ($(P1') + (DR)$);
	\coordinate (P3') at ($(P2') + (DR)$);
	\coordinate (P4') at ($(P3') + (DR)$);
	\coordinate (P5') at ($(P4') + (DR)$);
	\coordinate (Px') at ($0.5*(P3') + 0.5*(P4')$);
	\coordinate (P4) at (P5');
	\coordinate (P2) at ($(P4) + (UR)$);
	\coordinate (P3M) at ($(Px') + (DL)$);
	\coordinate (P32M) at ($(P3M) + 1.5*(DL)$);
	\coordinate (P2M) at ($(Px') + (UR)$);
	\draw[dashed, line width=0.1mm] ($(ZERO) + 4*(L) + 9*(D)$) -- ($(ZERO) + 4*(L) + 7*(U)$) -- ($(ZERO) + 8.5*(R) + 7*(U)$) -- ($(ZERO) + 8.5*(R) + 9*(D)$) -- ($(ZERO) + 4*(L) + 9*(D)$);
	\node[label={\huge $P'$}] (char) at ($(ZERO) + 7*(R) +4.5*(U)$) {};
	\draw[rotate=45, black!10!black, fill=black!10, line width=0.3mm, dotted] ($(Px') + 1.6*(DL)$) ellipse (1.06 and 0.6);
	\node[label=below:{\Large $D_1$}] (char) at ($(P3M)+2*(DL)$){};
	\draw[rotate=45, black, fill=black!10, line width=0.3mm, dotted] ($(P2M)+0.6*(UR)$) ellipse (1.06 and 0.6);
	\node[label=above:{\Large $U_k$}] (char) at ($(P2M)+2*(UR)$){};
	\draw[black, line width=0.2mm, dotted] ($(P1')+(0,1.1)$) ellipse (0.5 and 1);
	\node[label=above:{$U_1$}] (char) at ($(P1')+3.8*(U)$){};
	\draw[black, line width=0.2mm, dotted] ($(P5')-(0,1.1)$) ellipse (0.5 and 1);
	\node[label=below:{$D_k$}] (char) at ($(P5')+3.8*(D)$){};
	\draw[black!10!black, line width=1pt] (Px') -- ($(P3M)$);
	\draw[black!10!black, line width=1pt] (P1') -- ($0.5*(P2') + 0.5*(P3')$);
	\draw[black!10!black, dotted, line width=1pt] ($0.5*(P2') + 0.5*(P3')$) -- (Px');
	\draw[black, dotted, line width=1pt] (Px') -- ($0.5*(P4') + 0.5*(P5')$);
	\draw[black, line width=1pt] ($0.5*(P4') + 0.5*(P5')$) -- (P5');
	\draw[black, line width=1pt] (Px') -- (P2M);
	\node[circle, fill=black!10!black, inner sep=1pt, minimum size=5pt, label=225:{$p_3$}] (char) at (P3M) {};
	\node[circle, fill=black, inner sep=1pt, minimum size=5pt, label=45:{$p_2$}] (char) at (P2M) {};
	\node[circle, fill=black!10!black, inner sep=1pt, minimum size=5pt, label=left:{$p_1'$}] (char) at (P1') {};
	\node[circle, fill=black!10!black, inner sep=1pt, minimum size=5pt, label=45:{$p_2'$}] (char) at (P2') {};
	\node[circle, fill=black!10!black, inner sep=1pt, minimum size=5pt, label=45:{}] (char) at (Px') {};
	\node[circle, fill=black, inner sep=1pt, minimum size=5pt, label=right:{$p_k'$}] (char) at (P5') {};
\end{scope}
\end{tikzpicture}

		

%% file: tikz_k2_Pd.tex
\begin{tikzpicture}
\begin{scope}
	\coordinate (ZERO) at (0,0);
	\coordinate (R) at (0.5,0);
	\coordinate (L) at (-0.5,0);
	\coordinate (D) at (0,-0.5);
	\coordinate (U) at (0,0.5);
	\coordinate (UR) at (0.5,0.5);
	\coordinate (UL) at (-0.5,0.5);
	\coordinate (DR) at (0.5,-0.5);
	\coordinate (DL) at (-0.5,-0.5);
	\coordinate (P3) at ($(ZERO) + 2*(D) + (R)$);
	\coordinate (P1) at ($(P3) + (UR)$);
	\coordinate (P1') at (P1);
	\coordinate (P2') at ($(P1') + (DR)$);
	\coordinate (P4) at (P2');
	\coordinate (P2) at ($(P4) + (UR)$);
	\coordinate (P3M) at ($(P1') + (DL)$);
	\coordinate (P2M) at ($(P1') + (UR)$);
	\draw[rotate=45, black, fill=black!10, line width=0.3mm, dotted] ($(P1')+1.6*(DL)$) ellipse (1.5*1.41*0.5 and 0.6);
	\draw[rotate=45, black!10!red, fill=red!10, line width=0.3mm, dotted] ($(P1')+1.6*(UR)$) ellipse (1.5*1.41*0.5 and 0.6);
	\draw[black, line width=0.2mm, dotted] ($(P1')+(0,1.1)$) ellipse (0.5 and 1);
	\draw[black, line width=0.2mm, dotted] ($(P2')-(0,1.1)$) ellipse (0.5 and 1);
	\draw[dashed, line width=0.1mm] ($(ZERO) + 3*(L) + 8*(D)$) -- ($(ZERO) + 3*(L) + 6*(U)$) -- ($(ZERO) + 9*(R) + 6*(U)$) -- ($(ZERO) + 9*(R) + 8*(D)$) -- ($(ZERO) + 3*(L) + 8*(D)$);
	\node[label={\huge $P'$}] (char) at ($(ZERO) + 7.5*(R) +3.5*(U)$) {};
	\draw[black, line width=1pt] (P3M) -- (P1');
	\draw[black!10!red, line width=1pt] (P1') -- (P2');	
	\draw[black!10!red, line width=1pt] (P1') -- (P2M);			%
	\node[circle, fill=black, inner sep=1pt, minimum size=5pt, label=225:{$p_3$}] (char) at (P3M) {};
	\node[circle, fill=black!10!red, inner sep=1pt, minimum size=5pt, label=45:{$p_2$}] (char) at (P2M) {};
	\node[circle, fill=black!10!red, inner sep=1pt, minimum size=5pt, label=above:{}] (char) at (P1') {};
	\node[circle, fill=black!10!red, inner sep=1pt, minimum size=5pt, label=right:{$p_4$}] (char) at (P2') {};
	\node[label=above:{\Large $U_k$}] (char) at ($(P2M)+2*(UR)$){};
	\node[label=below:{\Large $D_1$}] (char) at ($(P3M)+2*(DL)$){};
	\node[label=above:{$U_1$}] (char) at ($(P1')+3.8*(U)$){};
	\node[label=below:{$D_k$}] (char) at ($(P2')+3.8*(D)$){};
\end{scope}
\end{tikzpicture}

%% file: tikz_k2_P.tex
\begin{tikzpicture}
\begin{scope}
	\coordinate (ZERO) at (0,0);
	\coordinate (R) at (0.5,0);
	\coordinate (L) at (-0.5,0);
	\coordinate (D) at (0,-0.5);
	\coordinate (U) at (0,0.5);
	\coordinate (UR) at (0.5,0.5);
	\coordinate (UL) at (-0.5,0.5);
	\coordinate (DR) at (0.5,-0.5);
	\coordinate (DL) at (-0.5,-0.5);
	\coordinate (P3) at ($(ZERO) + 2*(D) + (R)$);
	\coordinate (P1) at ($(P3) + (UR)$);
	\coordinate (P1') at (P1);
	\coordinate (P2') at ($(P1') + (DR)$);			\coordinate (P4) at (P2');
	\coordinate (P2) at ($(P4) + (UR)$);
	\draw[dashed, line width=0.1mm] ($(ZERO) + 3*(L) + 8*(D)$) -- ($(ZERO) + 3*(L) + 6*(U)$) -- ($(ZERO) + 9*(R) + 6*(U)$) -- ($(ZERO) + 9*(R) + 8*(D)$) -- ($(ZERO) + 3*(L) + 8*(D)$);
	\node[label={\huge $P$}] (char) at ($(ZERO) + 7.5*(R) +3.5*(U)$) {};
	\draw[rotate=45, black, fill=black!10, line width=0.3mm, dotted] ($(P3)+0.6*(DL)$) ellipse (1.5*1.41*0.5 and 0.6);
	\draw[rotate=45, black!10!red, fill=red!10, line width=0.3mm, dotted] ($(P2)+0.6*(UR)$) ellipse (1.5*1.41*0.5 and 0.6);
	\draw[black, line width=0.2mm, dotted] ($(P1')+(0,1.1)$) ellipse (0.5 and 1);
	\draw[black, line width=0.2mm, dotted] ($(P2')-(0,1.1)$) ellipse (0.5 and 1);
	\draw[black, line width=1pt] (P3) -- (P1');
	\draw[black, line width=1pt] (P1') -- (P2');		
	\draw[red, line width=1pt] (P2') -- (P2);
	\node[circle, fill=black!10!red, inner sep=1pt, minimum size=5pt, label=225:{$p_3$}] (char) at (P3) {};
	\node[circle, fill=black!10!red, inner sep=1pt, minimum size=5pt, label=45:{$p_2$}] (char) at (P2) {};
	\node[circle, fill=black, inner sep=1pt, minimum size=5pt, label=above:{}] (char) at (P1') {};
	\node[circle, fill=black!10!red, inner sep=1pt, minimum size=5pt, label=below:{}] (char) at (P2') {};
	\node[label=above:{\Large $U_k$}] (char) at ($(P2)+2*(UR)$){};
	\node[label=below:{\Large $D_1$}] (char) at ($(P3)+2*(DL)$){};
	\node[label=above:{$U_1$}] (char) at ($(P1')+3.8*(U)$){};
	\node[label=below:{$D_k$}] (char) at ($(P2')+3.8*(D)$){};
\end{scope}
\end{tikzpicture}

%% file: tikz_k3_Pd.tex
\begin{tikzpicture}
\begin{scope}
	\coordinate (ZERO) at (0,0);
	\coordinate (R) at (0.5,0);
	\coordinate (L) at (-0.5,0);
	\coordinate (D) at (0,-0.5);
	\coordinate (U) at (0,0.5);
	\coordinate (UR) at (0.5,0.5);
	\coordinate (UL) at (-0.5,0.5);
	\coordinate (DR) at (0.5,-0.5);
	\coordinate (DL) at (-0.5,-0.5);
	\coordinate (P3) at ($(ZERO) + (L)$);
	\coordinate (P32) at ($(P3) + 1.5*(DL)$);
	\coordinate (P1) at ($(P3) + (UR)$);
	\coordinate (P1') at (P1);
	\coordinate (P2') at ($(P1') + (DR)$);
	\coordinate (P3') at ($(P2') + (DR)$);
	\coordinate (P4') at ($(P3') + (DR)$);
	\coordinate (P5') at ($(P4') + (DR)$);
	\coordinate (Px') at ($0.5*(P3') + 0.5*(P4')$);
	\coordinate (P4) at (P5');
	\coordinate (P2) at ($(P4) + (UR)$);
	\coordinate (P3M) at ($(Px') + (DL)$);
	\coordinate (P32M) at ($(P3M) + 1.5*(DL)$);
	\coordinate (P2M) at ($(Px') + (UR)$);
	\draw[dashed, line width=0.1mm] ($(ZERO) + 4*(L) + 9*(D)$) -- ($(ZERO) + 4*(L) + 7*(U)$) -- ($(ZERO) + 8.5*(R) + 7*(U)$) -- ($(ZERO) + 8.5*(R) + 9*(D)$) -- ($(ZERO) + 4*(L) + 9*(D)$);
	\node[label={\huge $P'$}] (char) at ($(ZERO) + 7*(R) +4.5*(U)$) {};
	\draw[rotate=45, black!10!red, fill=red!10, line width=0.3mm, dotted] ($(Px') + 1.6*(DL)$) ellipse (1.06 and 0.7);
	\node[label=below:{\Large $D_1$}] (char) at ($(P3M)+2*(DL)$){};
	\draw[rotate=45, black, fill=black!10, line width=0.3mm, dotted] ($(P2M)+0.6*(UR)$) ellipse (1.5*1.41*0.5 and 0.6);
	\node[label=above:{\Large $U_k$}] (char) at ($(P2M)+2*(UR)$){};
	\draw[black, line width=0.2mm, dotted] ($(P1')+(0,1.1)$) ellipse (0.5 and 1);
	\node[label=above:{$U_1$}] (char) at ($(P1')+3.8*(U)$){};
	\draw[black, line width=0.2mm, dotted] ($(P5')-(0,1.1)$) ellipse (0.5 and 1);
	\node[label=below:{$D_k$}] (char) at ($(P5')+3.8*(D)$){};
	\draw[black!10!red, line width=1pt] (Px') -- ($(P3M) + 0.5*(DL)$);
	\draw[black!10!red, dotted, line width=1pt] ($(P3M) + 0.5*(DL)$) -- (P32M);
	\draw[black!10!red, line width=1pt] (P1') -- ($0.5*(P2') + 0.5*(P3')$);
	\draw[black!10!red, dotted, line width=1pt] ($0.5*(P2') + 0.5*(P3')$) -- (Px');
	\draw[black, dotted, line width=1pt] (Px') -- ($0.5*(P4') + 0.5*(P5')$);
	\draw[black, line width=1pt] ($0.5*(P4') + 0.5*(P5')$) -- (P5');
	\draw[black, line width=1pt] (Px') -- (P2M);
	\node[circle, fill=black!10!red, inner sep=1pt, minimum size=5pt, label=above:{$p_3$}] (char) at (P3M) {};
	\node[circle, fill=black!10!red, inner sep=1pt, minimum size=5pt, label=above:{}] (char) at (P32M) {};
	\node[circle, fill=black, inner sep=1pt, minimum size=5pt, label=45:{$p_2$}] (char) at (P2M) {};
	\node[circle, fill=black!10!red, inner sep=1pt, minimum size=5pt, label=left:{$p_1'$}] (char) at (P1') {};
	\node[circle, fill=black!10!red, inner sep=1pt, minimum size=5pt, label=45:{$p_2'$}] (char) at (P2') {};
	\node[circle, fill=black!10!red, inner sep=1pt, minimum size=5pt, label=45:{}] (char) at (Px') {};
	\node[circle, fill=black, inner sep=1pt, minimum size=5pt, label=right:{$p_k'$}] (char) at (P5') {};
\end{scope}
\end{tikzpicture}

		

%% file: tikz_k3_P.tex
\begin{tikzpicture}
\begin{scope}
	\coordinate (ZERO) at (0,0);
	\coordinate (R) at (0.5,0);
	\coordinate (L) at (-0.5,0);
	\coordinate (D) at (0,-0.5);
	\coordinate (U) at (0,0.5);
	\coordinate (UR) at (0.5,0.5);
	\coordinate (UL) at (-0.5,0.5);
	\coordinate (DR) at (0.5,-0.5);
	\coordinate (DL) at (-0.5,-0.5);
	\coordinate (P3) at (ZERO);
	\coordinate (P32) at ($(P3) + 1.5*(DL)$);
	\coordinate (P1) at ($(P3) + (UR)$);
	\coordinate (P1') at (P1);
	\coordinate (P2') at ($(P1') + (DR)$);
	\coordinate (P3') at ($(P2') + (DR)$);
	\coordinate (P4') at ($(P3') + (DR)$);
	\coordinate (P5') at ($(P4') + (DR)$);
	\coordinate (Px') at ($0.5*(P3') + 0.5*(P4')$);
	\coordinate (P4) at (P5');
	\coordinate (P2) at ($(P4) + (UR)$);
	\draw[dashed, line width=0.1mm] ($(ZERO) + 4*(L) + 9*(D)$) -- ($(ZERO) + 4*(L) + 7*(U)$) -- ($(ZERO) + 10*(R) + 7*(U)$) -- ($(ZERO) + 10*(R) + 9*(D)$) -- ($(ZERO) + 4*(L) + 9*(D)$);
	\node[label={\huge $P$}] (char) at ($(ZERO) + 8.5*(R) +4.5*(U)$) {};
	\draw[rotate=45, black, fill=black!10, line width=0.3mm, dotted] ($(P1') + 1.6*(DL)$) ellipse (1.06 and 0.7);
	\node[label=below:{\Large $D_1$}] (char) at ($(P3)+2*(DL)$){};
	\draw[rotate=45, black, fill=black!10, line width=0.3mm, dotted] ($(P2)+0.6*(UR)$) ellipse (1.5*1.41*0.5 and 0.6);
	\node[label=above:{\Large $U_k$}] (char) at ($(P2)+2*(UR)$){};
	\draw[black, line width=0.2mm, dotted] ($(P1')+(0,1.1)$) ellipse (0.5 and 1);
	\node[label=above:{$U_1$}] (char) at ($(P1')+3.8*(U)$){};
	\draw[black, line width=0.2mm, dotted] ($(P5')-(0,1.1)$) ellipse (0.5 and 1);
	\node[label=below:{$D_k$}] (char) at ($(P5')+3.8*(D)$){};
	\draw[black, line width=1pt] (P1) -- ($(P3) + 0.5*(DL)$);
	\draw[black, dotted, line width=1pt] ($(P3) + 0.5*(DL)$) -- (P32);
	\draw[black!10!red, line width=1pt] (P1') -- ($0.5*(P2') + 0.5*(P3')$);
	\draw[black!10!red, dotted, line width=1pt] ($0.5*(P2') + 0.5*(P3')$) -- ($0.5*(P4') + 0.5*(P5')$);
	\draw[black!10!red, line width=1pt] ($0.5*(P4') + 0.5*(P5')$) -- (P5');
	\draw[black, line width=1pt] (P4) -- (P2);
	\node[circle, fill=black, inner sep=1pt, minimum size=5pt, label=above:{$p_3$}] (char) at (P3) {};
	\node[circle, fill=black, inner sep=1pt, minimum size=5pt, label=above:{}] (char) at (P32) {};
	\node[circle, fill=black, inner sep=1pt, minimum size=5pt, label=45:{$p_2$}] (char) at (P2) {};
	\node[circle, fill=black!10!red, inner sep=1pt, minimum size=5pt, label=above:{}] (char) at (P1') {};
	\node[circle, fill=black!10!red, inner sep=1pt, minimum size=5pt, label=45:{$p_2'$}] (char) at (P2') {};
	\node[circle, fill=black!10!red, inner sep=1pt, minimum size=5pt, label=45:{$p_x'$}] (char) at (Px') {};
	\node[circle, fill=black!10!red, inner sep=1pt, minimum size=5pt, label=below:{}] (char) at (P5') {};
\end{scope}
\end{tikzpicture}

		